\documentclass[a4paper,11pt]{amsart}

\usepackage[latin1]{inputenc}

\usepackage{latexsym}
\usepackage{amssymb}
\usepackage{amsfonts}
\usepackage{amsmath}
\usepackage{amsthm}
\usepackage{textcomp}
\usepackage[usenames]{color}
\usepackage{comment}

\usepackage[colorlinks=true,urlcolor=blue,
citecolor=red,linkcolor=blue,linktocpage,pdfpagelabels,
bookmarksnumbered,bookmarksopen]{hyperref}
\usepackage{hyperref}

\usepackage {tikz}
\usetikzlibrary {positioning}
\definecolor {gray}{cmyk}{1,1,1,1}

\newtheorem{theorem}{Theorem}[section]    
 
 \newtheorem{corollary}[theorem]{Corollary}

\theoremstyle{definition}
\newtheorem{definition}[theorem]{Definition}     
\newtheorem{example}[theorem]{Example}
\newtheorem{remark}[theorem]{Remark}

\definecolor{grey}{rgb}{.7,.7,.7}

\title[A discrete districting plan]{A discrete districting plan}

\author[A. Saracco]{Alberto Saracco}
\address{Alberto Saracco\\ Dipartimento di Scienze Matematiche, Fisiche e Informatiche\\ Universit\`a di Parma\\ Parco Area delle Scienze 53/A\\ I-43124 Parma\\ Italy}
\email{alberto.saracco@unipr.it}

\author[G. Saracco]{Giorgio Saracco}
\address{Giorgio Saracco\\ Dipartimento di Matematica\\ Universit\`a degli Studi di Pavia\\ via Ferrata 5\\ I-27100 Pavia\\ Italy}
\email{giorgio.saracco@unipv.it}

\thanks{A.~Saracco was partially supported by INdAM-GNSAGA. G.~Saracco was partially supported by the INdAM-GNAMPA 2019 project ``Problemi isoperimetrici in spazi Euclidei e non''.}



\begin{document}

\begin{abstract}
The outcome of elections is strongly dependent on the districting choices, making thus possible (and frequent) the gerrymandering phenomenon, i.e.\ politicians suitably changing the shape of electoral districts in order to win the forthcoming elections. While so far the problem has been treated using continuous analysis tools, it has been recently pointed out that a more reality-adherent model would use the discrete geometry of graphs or networks. Here we propose a parameter-dependent discrete model for choosing an ``optimal'' districting plan. We analyze several properties of the model and lay foundations for further analysis on the subject.
\end{abstract}

 \hspace{-3cm}
 {
 \begin{minipage}[t]{0.6\linewidth}
 \begin{scriptsize}
 \vspace{-3cm}
 This is a pre-print of an article published in \emph{Netw. Heterog. Media}. The final authenticated version is available online at: \href{http://dx.doi.org/10.3934/nhm.2019031}{doi:10.3934/nhm.2019031}
 \end{scriptsize}
\end{minipage} 
}

\maketitle

\section{Introduction}

The most ancient reference to the isoperimetric problem is nowadays known by the name of Dido's problem which has a political background. Legend has it that queen Dido of Carthage was given the chance to found a city on the area she would have been able to enclose with a given ox hide: she cut the hide in thin strips and proceeded to enclose a very large area. If one were to seize this opportunity and get the most out of it, what would be the best shape? In such a context, by best we mean \textit{with the greatest area} given a length. Nowadays this problem is more commonly known through its dual formulation, i.e. to find the least perimeter enclosing a given area. It has been hypothesized for millennia that the best is given by the circle, yet a formal proof was not available until very recently. A first step toward the proof was made by Steiner in the 19th century who showed that \textit{if a solution existed, it had to be the circle}~\cite{Ste38}. It was in the 20th century that De Giorgi bridged the last gaps proving the full result via the theory of sets of finite perimeter and BV functions he developed in~\cite{DeG54}.\par

Whenever one has to minimize the (some notion of) perimeter with a constraint on the (some notion of)  area, one refers to the problem as an isoperimetric-like.
These problems have a wide number of practical applications ranging from physics such as the capillarity phenomenon~\cite{LS18b, LeoSar18} to crystallography such as the shapes of equilibrium crystals~\cite{Cer06}. In the latest years isoperimetric-like problems were looked at with interest from social sciences with the explicit aim to recognize gerrymandering phenomena in politics~\cite{DucTen18, GMPRSbook}.\par

Most of the literature on gerrymandering discusses the current shapes of electoral districts and assigns to each different scores, all of isoperimetric nature whose goal is to measure the ``compactness'' of the shapes, whatever this may mean. The aim is twofold: compare districts via these scores and rank them from the most gerrymandered to the least, see for instance~\cite{BS18}; decide whether a district is gerrymandered or not, see for instance~\cite{BVGHSKLMR17,HKLVgBRM18, HRM17, MVG14}.\par

In this paper we propose a discrete districting plan and discuss some properties we would like our model to possess. The underlying idea is that given some region $\Omega$ we want to partition it in $N$ subregions $\{\Omega_k\}_{k=1}^N$ with the same population, i.e.
\begin{equation}\label{eq:ContConstraint}
\int_{\Omega_k} f(x)\, dx = \frac 1N \|f\|_{L^1(\Omega)}
\end{equation}
where $f$ denotes the population density over $\Omega$. A general approach to get the ``most compact'' shape is to minimize the (relative) perimeter of the partition, thus we would be led to minimize the functional
\[
\frac 12 \sum_{k=1}^N P(\Omega_k; \Omega)\,,
\]
subject to~\ref{eq:ContConstraint}, which is linked to the score proposed in~\cite{FLSS18}. Though, one might argue that it would be best to weigh the perimeter as well via the density population in order to make uneconomical to split highly populated regions. Indeed a high density might be an indicator of strong political, ethnical, religious, linguistic and so on ties, thus it would be fair to not split them. Then, the functional to be minimized would be
\begin{equation*}\label{eq:ContPer1}
\frac 12 \sum_{k=1}^N \int_{\Omega \cap \partial^* \Omega_k} f(x) \, d\,\mathcal{H}^{n-1}(x)\,,
\end{equation*}
where $\partial^* \Omega_k$ denotes the reduced boundary. More general functionals to be minimized can be proposed such as
\begin{equation}\label{eq:ContPer2}
\frac 12 \sum_{k=1}^N \int_{\Omega \cap \partial^* \Omega_k} g(x, \nu_{\Omega_k}(x)) \, d\,\mathcal{H}^{n-1}(x)\,,
\end{equation}
where the weight $g=g(x,\nu)$ takes into account not only the point but as well the direction of the boundary at the point. This weight $g$ might represent how big the flow of people and the exchange of informations at point $x$ in direction $\nu$ is. Isoperimetric-like problems of this fashion have been treated in~\cite{PraSar18, Sar18}, as well as the regularity of minimizers in~\cite{PraSar19}.\par

Continuous descriptions of the gerrymandering phenomenon suffer though from many issues as highlighted in~\cite{DucTen18} where the authors propose to make use of tools from discrete geometry and the theory of graphs to depict the scenario: some models taking into account a graph structure have been studied for instance in~\cite{ABLRS09} (see also~\cite{RSS13}). Indeed, the situation may be very well described in the setting of discrete geometry as the total population is finite. The idea is to choose one of the levels of the census' units (for instance US: block, block's group, tract; Italy: comune, provincia, regione) and assign to each a vertex $v_i$ in a graph $\Gamma$. We shall then say that two vertexes, $v_i, v_j$, are adjacent if the corresponding units share (a positive amount of) boundary and denote their edge by $e_{i,j}$. This latter hypothesis ensures that the resulting graph is planar. Then, the weight $f$ denotes the population of a vertex, while the weight $g$ is a measure of how well two adjacent vertexes are connected. Hence, the discrete functional  equivalent to~\ref{eq:ContPer2} is
\begin{equation}\label{eq:DiscPer}
 \sum_{\mathcal{C}(\{\Gamma_k\})} g(e)\,,
\end{equation}
where $\{\Gamma_k\}$ is a $N$-partition of $\Gamma$ and the sum is taken over the edges $e$ belonging to the \emph{cut set} $\mathcal{C}(\{\Gamma_k\})$, i.e.\ sloppily speaking the edges removed from $\Gamma$ to obtain the partition $\{\Gamma_k\}$ (see Definition~\ref{def:partition} for the formal details of partition and cut set). Clearly one can not impose a constraint analogous to~\ref{eq:ContConstraint} as in general this would prevent a solution to exist. Adding a penalization term of the form
\begin{equation}\label{eq:DiscConstr}
\sqrt{\sum_{k=1}^N \left(\sum_{v \in \Gamma_k} f(v) - \frac{\sum_{v \in \Gamma} f(v)}{N} \right)^2}\,,
\end{equation}
rather than a mass constraint, ensures the existence of solutions without completely dropping the request to have (almost) equally populated regions. This term represents the standard deviation of the populations in each subgraph with respect to the mean population of the whole graph, and one could use a general $p$-norm in place of the $2$-norm. The functional we will look at is the convex combination of~\ref{eq:DiscPer} and~\ref{eq:DiscConstr}, i.e.
\[
\lambda \sum_{\mathcal{C}(\{\Gamma_k\})} g(e) + (1-\lambda)\sqrt{\sum_{k=1}^N \left(\sum_{v \in \Gamma_k} f(v) - \frac{\sum_{v \in \Gamma} f(v)}{N} \right)^2}\,.
\]
Depending on the choice of $\lambda$ one gives more or less prominence to one of the two terms. When one looks purely at the perimeter energy~\ref{eq:DiscPer}, i.e. for $\lambda=1$, the problem is closely related with the one known as minimum-$N$-cut, where one seeks to split the graph in \emph{at least} $N$ components rather than \emph{exactly} $N$. \par

The proposed functional clearly does not only have applications to politics. One can imagine many different scenarios: for instance the vertexes of the graph might represent computers with the weight $f$ their powers, while edges represent direct connections and the weight $g$ how many MB/s of data-flow these links grant. Then our problem would represent the need to break this computer network in smaller groups of similar power by cutting the slowest connections. In this latter case though it is very possible that the starting graph is not planar, possibly adding more algorithmic complexity. For the sake of completeness, we recall that for fixed $N$ the minimum-$N$-cut problem is polynomial time solvable (see~\cite{BMSSS16, GH94} and the references therein).\par

The paper is organized as follows. In Section~\ref{sec:model} we lay the notation and precisely define the functional. In Section~\ref{sec:properties} we discuss a series of properties one would like to have for such a problem. For each desired property we provide either a proof of the property or show a counterexample disproving it. In Section~\ref{sec:p-norms} we briefly discuss what happens if we consider a more general deviation energy term. In Section~\ref{sec:conclusions} we discuss the open problems we have not settled yet, which would be nice to explore, and define some further research directions we shall investigate in the future.

\section{The model}\label{sec:model}
Let us consider a graph
\[
\Gamma=\big( V(\Gamma) ; E(\Gamma) \big)\,.
\]
The elements $\{v_i\}_i = V(\Gamma)$ represent the vertexes of the graph, while those of $E(\Gamma)$ the edges, i.e. the connections $e^k_{i,j}$ from vertex $v_i$ to vertex $v_j$ available in the graph. We shall suppose that
\begin{itemize}
\item $\Gamma$ is finite, i.e. $|V(\Gamma)| < +\infty$;
\item $\Gamma$ is simple, i.e. there is at most one edge $e_{i,j}$ connecting the vertex $v_i$ to $v_j$ and there are no loops $e_{i,i}$;
\item $\Gamma$ is undirected, i.e. $e_{i,j}$ is identified with $e_{j,i}$;
\item $\Gamma$ is connected, i.e. given any two vertexes $v, w$ there exists a sequence of vertexes $\{v_k\}_{k=1}^n$ such that $v_1 = v$, $v_n = w$ and $e_{k, k+1}\in E(\Gamma)$ for all $k=1,\dots, n-1$.
\end{itemize}
We endow the graph $\Gamma$ with two different weights, one acting on the vertexes and one on the edges. Specifically, we set $f:V(\Gamma) \to (0,+\infty)$ and $g:E(\Gamma) \to (0,+\infty)$. Graphs with weights are occasionally referred to as \emph{networks}. More details and basic notions of graphs can be found in the monographs~\cite{BoMu76, Wil96}. Given these choices, one could represent the weighted edges as the symmetric traceless square matrix $A\in \mathbb{M}(|V(\Gamma)|)$, where $a_{i,j} = g(e_{i,j})$ if $e_{i,j}\in E(\Gamma)$ and $a_{i,j}=0$ otherwise.\par

The graph amounts to the region which we want to divide into electoral districts; the vertex $v_i$ represents a town or neighborhood in the region ($f(v_i)$ being the population of $v_i$) and the edge $e_{i,j}$ represents a direct connection between two nearby towns/neighborhoods. The weight $g(e_{i,j})$ is a measure of how good the connection between the two towns is: the greater $g(e_{i,j})$, the better they are connected, e.g. $g(e_{i,j})$ represents the number of people that can go from one town to the other in a fixed amount of time.\par

We denote by
\[
M(\Gamma)\ =\ \sum_{v\in\Gamma}f(v)
\]
the total population, or \emph{mass}, of a graph.
\begin{definition}
Given a graph $\Gamma=\big( V(\Gamma) ; E(\Gamma) \big)$, we say that $\Gamma_k$ is a subgraph of $\Gamma$ if it is a graph such that $V(\Gamma_k)\subseteq V(\Gamma)$ and $E(\Gamma_k) \subseteq E(\Gamma)$. Moreover, we define its boundary $\partial \Gamma_k$ as
\[
\partial \Gamma_k := \left\{e_{i,j}\in E(\Gamma)\,:\, v_i \in V(\Gamma_k)\,, v_j \notin V(\Gamma_k)\right\}\,.
\]
\end{definition}

In our model we are interested in pairwise disjoint and connected partitions of the graph. For the sake of brevity, in the following we shall only say $N$-partition, truly referring to a pairwise disjoint, connected $N$-partition accordingly to the next definition.
\begin{definition}\label{def:partition}Let $1\,\leq\, N\,\leq\,|V(\Gamma)|$ be an integer. A \textit{pairwise disjoint and connected $N$-partition} of $\Gamma$ is a family of $N$ connected subgraphs $\{\Gamma_k\}_{k=1}^N$ such that
\[
\bigcup_k V(\Gamma_k) = V(\Gamma)\,, \qquad V(\Gamma_k) \cap V(\Gamma_h) = \emptyset\,, \forall k\neq h\,.
\]
We define the \emph{cut set} of the partition (or boundary of the partition) as
\[
\mathcal{C}(\{\Gamma_k\}):= \bigcup_k \partial \Gamma_k\,.
\]
\end{definition}

Given $N$, we would like to find a $N$-partition of $\Gamma$ such that all subgraphs $\Gamma_k$ have the same total weight of the vertex, i.e. the same mass
\[
M(\Gamma_k) = \sum_{V(\Gamma_k)}f(v)\ =\ \frac{1}{N} M(\Gamma)\,.
\]
In our electoral interpretation this means dividing the region $\Gamma$ in $N$ electoral districts with the same population, while minimizing the total weight of the cut set
\begin{equation}\label{eq:P-energy}
P(\{\Gamma_k\}):=\sum_{\mathcal{C}(\{\Gamma_k\})} g(e)
\end{equation}
to ensure that the districts are as ``compact'' as possible. We shall refer to the above as to the \emph{cut} or \emph{perimeter} energy of the partition. This of course is not possible since in general the set of $N$-partitions with all subgraphs with the same weight is empty. Thus, we want to allow the possibility for the weight of the subgraphs to differ from the arithmetic mean, at a cost, increasing with the difference from the mean. More precisely, we add the penalization term
\begin{equation}\label{eq:V-energy}
{\sigma(\{\Gamma_k\})}:= \sqrt{\sum_{k=1}^{N} \left(M(\Gamma_k) - \frac{M(\Gamma)}{N} \right)^2}\,,
\end{equation}
which we shall refer to as the \emph{deviation} energy of the partition.\par

Given $\lambda \in [0,1]$, we define the energy functional
\begin{equation}\label{eq:L-energy}
\mathcal{F}_\lambda(\{\Gamma_k\}) =\lambda P(\{\Gamma_k\}) + (1-\lambda)\sigma(\{\Gamma_k\})\,,
\end{equation}
which is a convex combination of the cut energy~\ref{eq:P-energy} and of the deviation energy~\ref{eq:V-energy}. We are interested in minimizing $\mathcal{F}_\lambda$ among all $N$-partitions of $\Gamma$. Notice that the minimization problem is invariant under the action of a uniform dilation of the weights $f$ and $g$, i.e. taking as weights $\theta f$ and $\theta g$, with $\theta >0$, in place of $f$ and $g$ does not change the minimizers. We shall call any partition minimizing~\ref{eq:L-energy} a \emph{minimal} or \emph{optimal} $N$-partition (relative to some $\lambda$).\par

It is worth noticing that for $\lambda=1$ the only energy we consider is the cut energy, i.e. the problem is similar to the minimal N-cut problem and there is no request on the mass of the districts to be near to the mean value.\par

\begin{remark}\label{rem:existence}
Existence of minimizers is trivial and follows straightforwardly from the finiteness of the graph $\Gamma$. Moreover, notice the following fact. Given $\hat \Gamma = \big(V(\hat \Gamma) ; E(\hat \Gamma)\big)$ a subgraph belonging to a partition of $\Gamma$, one can  suppose wlog that $E(\hat \Gamma) = \{e_{i,j}\in E(\Gamma)\,:\, v_i, v_j \in V(\hat \Gamma)\}$. This is because the first term of the functional $\mathcal{F}_\lambda$, i.e. the cut energy, is defined on the cut set $\mathcal{C}(\{\Gamma_k\})$ which is contained but not necessarily equal to $E(\Gamma)\setminus \cup_k E(\Gamma_k)$. Thus, the functional does not detect any internal changes in a subgraph $\hat \Gamma$, i.e. considering $\hat \Gamma$ or $\tilde \Gamma = (V(\hat \Gamma), E(\hat \Gamma)\setminus \{e\})$ is the same, provided that removing $e$ does not disconnect $\hat \Gamma$. This means that the energy of a partition $\{\Gamma_k\}$ is actually a function of the partition of the vertexes $\{V(\Gamma_k)\}$ and does not depend on the edges.
\end{remark}

Regarding uniqueness, note that fixed a $N$-partition its energy $\mathcal{F}_\lambda$ is affine linear in $\lambda$. Hence, it easily follows a uniqueness theorem.
\begin{theorem}\label{thm:uniqueness}
If a $N$-partition $\{\tilde \Gamma_k\}$ minimizes $\mathcal{F}_\lambda$ for two distinct values $\lambda_1 < \lambda_2$, then it is a minimal $N$-partition for all $\lambda \in [\lambda_1, \lambda_2]$. Moreover,
\begin{itemize}
\item[i)] either it is the unique minimizer in $(\lambda_1, \lambda_2)$;
\item[ii)] or there exists another partition $\{\hat \Gamma_k\}$ such that
\[
\mathcal{F}_\lambda (\{\hat \Gamma_k\}) = \mathcal{F}_\lambda (\{\tilde \Gamma_k\})\,, \qquad \forall \lambda \in [0,1]\,.
\]
\end{itemize}
\end{theorem}

\begin{proof}
Given $\{\tilde \Gamma_k\}$ minimal in $\lambda_1$ and $\lambda_2$ ($\lambda_1<\lambda_2$) we have that
\begin{equation}\label{eq:ipotesi}
\mathcal{F}_{\lambda_1}(\{\tilde \Gamma_k\}) \le  \mathcal{F}_{\lambda_1}(\{\Gamma_k\})\,, \qquad \mathcal{F}_{\lambda_2}(\{\tilde \Gamma_k\}) \le  \mathcal{F}_{\lambda_2}(\{\Gamma_k\})\,, \, \forall \{\Gamma_k\}\,.
\end{equation}
Notice now that the energy of any given partition is affine linear in $\lambda$ and thus determined by its value in any two given points, or equivalently $\mathcal{F}_\lambda$ is a linear combination of $\mathcal{F}_{\lambda_1}$ and $\mathcal{F}_{\lambda_2}$. Hence by~\ref{eq:ipotesi}, the energy of any partition $\{\Gamma_k\}$ is greater than or equal to the energy of $\{\tilde \Gamma_k\}$ in $[\lambda_1, \lambda_2]$.\par

Suppose now there exists a partition $\{\hat \Gamma_k\}$ such that $\mathcal{F}_{\bar \lambda}(\{\tilde \Gamma_k\}) =  \mathcal{F}_{\bar \lambda}(\{\hat \Gamma_k\})$ for some $\bar \lambda \in (\lambda_1, \lambda_2)$. Since the energy is affine linear, inequalities~\ref{eq:ipotesi} paired with this last equality necessarily imply that the two partitions always have the same energy. The converse is readily achieved by using again the linearity w.r.t. $\lambda$ and the hypothesis that at least one of the inequalities in~\ref{eq:ipotesi} is strict.
\end{proof}

\begin{corollary}
If there are no partitions that have the same energy in any two points, then $\mathcal{F}_\lambda$ has a unique minimizer up to a finite set of transition values.
\end{corollary}

\section{Discussion of (desirable) properties}\label{sec:properties}

We analyze here several properties we would like our districting model to have. We are interested in:
\begin{itemize}
\item[(i)] stability of the division in districts when a new town/neighborhood is built or abandoned;
\item[(ii)] stability of the division in districts when a new road connecting different cities is built (or an existing connection is destroyed);
\item[(iii)] possibility to force certain adjacent cities to be in the same district (up to suitably modifying the connection between the two);
\item[(iv)] possibility to force a city to form a district on its own (up to suitably modifying the connections between that city and its neighbors);
\item[(v)] stability of the model at a multi-layer districting, i.e.\ does creating ``super''-districts, say $N$, and then splitting them, say in $j$ ``sub''-districts each, yield the same result as directly creating $jN$ districts?
\end{itemize}

Our model does not grant properties (i), (ii), (v), as we show by counterexamples. Properties (iii) and (iv) hold true, under suitable conditions. The fact that properties (i), (ii) and (v) do not hold can be: either a weakness of the model deriving from the high degree of freedom of weights and may possibly be circumvented by modifying the form of the deviation energy; or an unavoidable drawback present in any possible model of districting. The latter is far from being unrealistic and it is actually something that often happens in social choices, e.g.
\begin{itemize}
\item Arrow in~\cite{Arr} proved that, if there are at least 3 choices, no electoral system satisfies at once Pareto's property (if all voters prefer $X$ over $Y$, $X$ is group-preferred to $Y$), independence from irrelevant choices (the group preference between $X$ and $Y$ only depends  on the single preferences between $X$ and $Y$) and there is no dictator (no single voter possesses the power to always determine the group's preference).
\item Balinski and Young in~\cite{BalYou} proved that, if there are at least 3 parties, no apportionment system simultaneously follows the quota rule (if a party fair share is between $n$ and $n+1$ it gets assigned either $n$ or $n+1$ seats), avoids the Alabama paradox (if the total number of seats is increased, no party's number of seats decreases) and avoids the population paradox (if party $A$ gets more votes and party $B$ gets fewer, no seat will be transferred from $A$ to $B$).
\end{itemize}

\subsection{Adding or removing a vertex}

We here briefly discuss what can happen whenever a vertex is removed from a graph (resp. added to). In our practical example of districting this could correspond to a city being abandoned  (resp. to a new city being built).

\begin{example}\label{+-vertex}We consider the graph in Figure~\ref{fig:+-vertex}, where the grayed-out vertex is the new/removed vertex and the dashed edges are the edges connecting it to the other vertexes, and look at its minimal $2$-partitions. We call $\Gamma$ the whole graph and $\hat \Gamma$ its subgraph without the grayed-out vertex and the dashed edges.\par

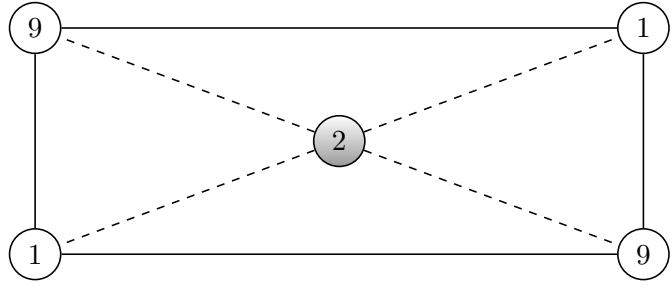
\begin{figure}[t]
\begin{center}
\begin {tikzpicture}[auto ,node distance =1.5 cm and 4cm ,on grid ,
semithick ,
state/.style ={ circle ,top color =white , bottom color = gray!00 ,
draw,gray , text=black , minimum width =0.5 cm}]
\node[state,  bottom color = gray!20 ] (1) {$2$};
\node[state] (2) [above right =of 1] {$1$};
\node[state] (3) [below right =of 1] {$9$};
\node[state] (4) [above left =of 1] {$9$};
\node[state] (5) [below left =of 1] {$1$};
\path (1) edge [dashed, bend right = 0]   (2);
\path (1) edge [dashed, bend right = 0]  (3);
\path (1) edge [dashed, bend right = 0]   (4);
\path (1) edge [dashed, bend right = 0]   (5);
\path (2) edge [bend left = 0]  (4);
\path (3) edge [bend left = 0]  (5);
\path (4) edge [bend right = 0]   (5);
\path (2) edge [bend right = 0]  (3);
\end{tikzpicture}
\caption{Removing or adding a vertex; numbers correspond to the weight of the vertexes; all edges are supposed to have weight $1$.}
\label{fig:+-vertex}
\end{center}
\end{figure}

We start by analyzing the graph \emph{without} the grayed-out vertex, $\hat \Gamma$. Notice that for $\lambda=1$ the minimal $2$-cut can be obtained by cutting edges with a total weight of $2$ (i.e. any two edges), while for ${\lambda=0}$ there are two partitions with null energy (cutting either the two horizontal edges, or the two vertical edges). As the partitions obtained by removing two parallel edges are minimal both for ${\lambda =0}$ and $\lambda = 1$ it follows from Theorem~\ref{thm:uniqueness} that they are minimal for all values of $\lambda$ and as well unique for $\lambda\in(0,1)$ as no other minimizer has the same energy on both the extrema of the interval.\par

As for the graph \emph{with} the grayed-out vertex, $\Gamma$, there are three different partitions to be considered, up to obvious symmetries:
\begin{itemize}
\item $D$, the partition of least deviation energy, i.e.\ the one for which one vertex with weight $9$ is in a district with the new vertex: the cut energy is $5$, while the deviation energy is null for which
\[
\mathcal{F}_\lambda(D)=5\lambda\,;
\]
\item $C$, the partition of least cut energy, i.e.\ the one for which one vertex with weight $9$ is a district on itself: the cut energy is $3$, while the deviation energy is $2\sqrt2$  for which
\[
\mathcal{F}_\lambda(C)=(3-2\sqrt2)\lambda+2\sqrt2\,;
\]
\item $M$, the one for which one vertex with weight $9$ is in a district with a vertex of weight $1$: the cut energy is $4$, while the deviation energy is $\sqrt2$  for which
\[
\mathcal{F}_\lambda(M)=(4-\sqrt2)\lambda+\sqrt2\,.
\]
\end{itemize}
The energy of the other possible $2$-partitions is strictly controlled from below by one of the above away from the extremal points, hence they can be a minimal $2$-partition only for $\lambda=0,1$. Thus the optimal partition is:
\begin{itemize}
\item $D$ if $0\leq\lambda<2-\sqrt2$;
\item $D$, $M$ and $C$ are equivalent if $\lambda=2-\sqrt2$;
\item $C$ if $2-\sqrt2<\lambda\leq1$.
\end{itemize}
Hence, whenever $0\leq\lambda<2-\sqrt2$ the minimal 2-partitions for $\Gamma$ and $\hat\Gamma$ are significantly different.\end{example}

\subsection{Adding or removing an edge}

We here briefly discuss what can happen whenever an edge is removed from a graph (resp. added to). In our practical example of districting this could correspond to a bridge collapsing  (resp. to a new road being built).

\begin{example}\label{+-edge} We consider the graph in Figure~\ref{fig:+-edge}, where the dashed edge is the new/removed edge, and look at its minimal $3$-partitions. If we call $\Gamma$ the whole graph and $e$ the dashed edge, we are then looking, respectively, at $\Gamma$ and its subgraph ${\hat \Gamma = \big( V(\Gamma); E(\Gamma)\setminus \{e\}\big)}$.\par

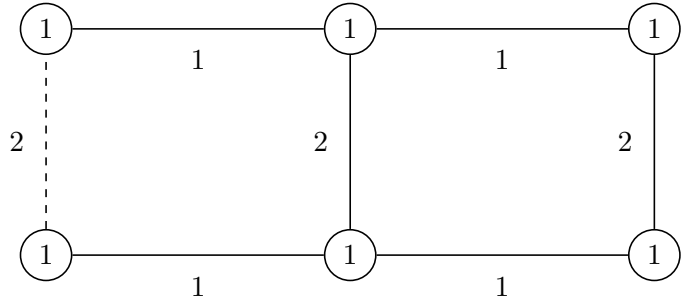
\begin{figure}[t]
\begin{center}
\begin {tikzpicture}[auto ,node distance =3 cm and 4cm ,on grid ,
semithick ,
state/.style ={ circle ,top color =white , bottom color = gray!00 ,
draw,gray , text=black , minimum width =.5 cm}]
\node[state] (1) {$1$};
\node[state] (2) [right =of 1] {$1$};
\node[state] (3) [right =of 2] {$1$};
\node[state] (4) [above =of 1] {$1$};
\node[state] (5) [above =of 2] {$1$};
\node[state] (6) [above =of 3] {$1$};
\path (1) edge [dashed, bend right = 0] node[left =0.15 cm] {$2$} (4);
\path (1) edge [bend right = 0] node[below =0.15 cm] {$1$} (2);
\path (4) edge [bend right = 0] node[below =0.15 cm] {$1$} (5);
\path (2) edge [bend right = 0] node[left =0.15 cm] {$2$} (5);
\path (2) edge [bend right = 0] node[below =0.15 cm] {$1$} (3);
\path (5) edge [bend right = 0] node[below =0.15 cm] {$1$} (6);
\path (3) edge [bend right = 0] node[left =0.15 cm] {$2$} (6);
\end{tikzpicture}
\caption{Removing or adding an edge; numbers correspond to the weight of the related vertexes and edges.} \label{fig:+-edge}
\end{center}
\end{figure}

We start by analyzing the graph \emph{with} the dashed edge, $\Gamma$. Notice that for $\lambda=1$ the minimal $3$-cut can be obtained by cutting edges with a total weight of $4$ (the three right edges, the three left ones or the four horizontal ones), while for $\lambda=0$ there are partitions with zero energy (cutting the 4 horizontal edges, or the left and center vertical edges and the two horizontal right edges, or the right and center vertical edges and the two horizontal left edges). As the partition obtained by removing the $4$ horizontal edges is minimal both for $\lambda =0$ and $\lambda = 1$ it follows from Theorem~\ref{thm:uniqueness} that it is minimal for all values of $\lambda$ and as well unique for $\lambda\in(0,1)$ as no other minimizer has the same energy on both the extrema of the interval. This $3$-partition though induces a $4$-partition for the subgraph $\hat \Gamma$ thus it can not be a minimizer of $\mathcal{F}_\lambda$ on $\hat \Gamma$ for any $\lambda$! In particular the minimal solutions relative to $\hat \Gamma$ are the following:
\begin{itemize}
\item for $0\leq\lambda \le 3-\sqrt{6}$ it is obtained by cutting the two horizontal left edges ($\mathcal{F}_\lambda=(2-\sqrt6)\lambda+\sqrt6$);
\item for $3-\sqrt{6}\le \lambda\leq1$ it is obtained by cutting the center vertical edge and the two horizontal right edges ($\mathcal{F}_\lambda=4\lambda$).
\end{itemize}
\end{example}

\subsection{Forcing two vertexes in the same component}
Let $\Gamma$ be a graph and $v_1$, $v_2$ be any two neighboring vertexes. There are obviously graphs where it is impossible to have $v_1$ and $v_2$ in the same district (e.g. if there are $N$ vertexes and we are looking for a $N$-partition). Yet, if there are $N$-partitions where these two vertexes are in the same district, one can force the minimal partition to be one of them, by a suitable modification of $g$ on the edge $e_{1,2}$. More precisely, the following result holds.

\begin{theorem}\label{2vertexes}Let $\Gamma$ be a graph and $v_1$, $v_2$ be two adjacent vertexes via the edge $e_{1,2}$ such that there exists a $N$-partition of $\Gamma$ for which $v_1$ and $v_2$ belong to the same subgraph. For any fixed $\lambda\neq0$, one can modify the value of $g$ on the edge $e_{1,2}$ in such a manner that $v_1$ and $v_2$ belong to the same subgraph of the $N$-partition minimizing $\mathcal{F}_\lambda$.
\end{theorem}

\begin{proof}Let us divide the set $\Gamma_N$ of $N$-partitions of $\Gamma$ as
\[
\Gamma_N\ = \ \Gamma_{N,+}\cup\Gamma_{N,-}\,,
\]
where $\Gamma_{N,+}$ is the set of $N$-partitions putting $v_1$ and $v_2$ in the same district, and $\Gamma_{N,-}$ is the set of $N$-partitions putting $v_1$ and $v_2$ in different districts.\par

Let us suppose $\Gamma_{N,-}\neq\emptyset$, otherwise the claim is trivial. By hypothesis one has as well $\Gamma_{N,+}\neq\emptyset$. The $\lambda$-energy of any $N$-partition of $\Gamma_{N,+}$ does not depend on the value of $g(e_{1,2})$. As the  number of partitions is finite, we can define $K_\lambda$ to be the minimum of those energies.\par

Notice now that the $\lambda$-energy of any $N$-partition $\{\Gamma_k\} \in \Gamma_{N,-}$ can be expressed as
\[
\mathcal{F}_\lambda(\{\Gamma_k\})\ =\ \lambda g(e_{1,2})+D_\lambda(\{\Gamma_k\})\,,
\]
where $D_\lambda(\{\Gamma_k\})$ does not depend on the value of $g(e_{1,2})$. Again in virtue of the finiteness of the partitions, let $D_\lambda$ be the minimum of $D_\lambda(\{\Gamma_k\})$, for $\{\Gamma_k\}\in\Gamma_{N,-}$.\par

Thus, by modifying $g$ such that $g(e_{1,2})>\lambda^{-1}  (K_\lambda-D_\lambda)$ one gets that the minimal partition belongs to $\Gamma_{N,+}$.
\end{proof}

Notice that the above proof only works for a single pair and there is no clear way to force multiple pairs to stick together as the functional $\mathcal{F}_\lambda$ displays a non-local behaviour with respect to the weights.

\begin{remark}\label{split}
Conversely, if one were to try to split two adjacent vertexes, provided that there exists a suitable partition, would find as a condition $g(e_{1,2})<\lambda^{-1}  (K_\lambda-D_\lambda)$, which is impossible to achieve if the RHS is non positive. There are indeed very easy cases for which this exact behaviour occurs. Take for instance the graph $\Gamma$ in Figure~\ref{fig:split}, where $M>1$. For any $\lambda < \sqrt{2}(1+\sqrt2)^{-1}$ the leftmost vertexes will stick together no matter how big $M$ is or how small $\varepsilon$ is.
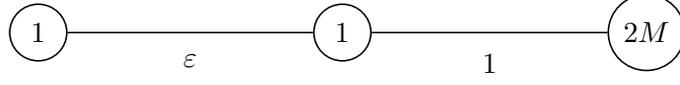
\begin{figure}[t]
\begin{center}
\begin {tikzpicture}[auto ,node distance =3 cm and 4cm ,on grid ,
semithick ,
state/.style ={ circle ,top color =white , bottom color = gray!00 ,
draw,gray , text=black , minimum width =.75 cm}]
\node[state] (1) {$1$};
\node[state] (2) [right =of 1] {$1$};
\node[state] (3) [right =of 2] {$2M$};
\path (1) edge [bend right = 0] node[below =0.15 cm] {$\varepsilon$} (2);
\path (2) edge [bend right = 0] node[below =0.15 cm] {$1$} (3);
\end{tikzpicture}
\caption{Splitting two vertexes is not always possible by simply modifying the weight of their common edge.} \label{fig:split}
\end{center}
\end{figure}
\end{remark}

\subsection{Forcing a vertex to be isolated}

By a straightforward application of the pigeonhole principle we can observe the following fact. Suppose that a graph $\Gamma$ has vertex connectivity of $1$ about the vertex $\bar v$, i.e., the subgraph $\hat \Gamma$ obtained by removing from $\Gamma$ the vertex $v$ and all related edges is disconnected. If $\hat \Gamma$ has  $k= N+l$ connected components, $l\in\mathbb{N}$, then in any $N$-partition of $\Gamma$ at least $l+1$ of these components belong to the subgraph containing $v$. In such a situation thus there is no way to force $v$ to be ``isolated'' i.e. to form a district on its own.

It is then of interest, assuming the necessary assumption that there exist $N$-partitions containing the singleton vertex $\Gamma_v:=\big( \{v\}; \emptyset \big)$, to know if one can force the vertex to be isolated by modifying the weights of the edges in $\partial \Gamma_v$. We are able to prove that this is possible for values of $\lambda$ near to $1$ limitedly to $2$-partitions. More precisely, the following theorem holds.
\begin{theorem}\label{thm:isolated}
Let $v\in V(\Gamma)$ be a vertex, $\Gamma_v := \big( \{v\}; \emptyset \big)$ the subgraph consisting of the singleton vertex and $\Gamma_v^c := \big(V(\Gamma)\setminus \{v\}; E(\Gamma)\setminus \partial \Gamma_v \big)$ its complement subgraph. If $\{\Gamma_v, \Gamma^c_v \}$ is a $2$-partition, there exist $\bar \varepsilon>0$ and $\bar \lambda \in [0,1)$ such that if $\sum_{\partial \Gamma_v}g(e) \le \bar \varepsilon$ then, $\{\Gamma_v, \Gamma^c_v \}$ is the minimal $2$-partition (up to removing edges of $\Gamma^c_v$ which do not disconnect it) of $\mathcal{F}_\lambda$ with $\lambda \in (\bar \lambda, 1]$.
\end{theorem}

\begin{proof}
Start noticing that if $\{\Gamma_v, \Gamma^c_v \}$ is a $2$-partition, then it is the unique $2$-partition containing the subgraph $\Gamma_v$ (up to removing edges of $\Gamma^c_v$ which do not disconnect it as noted in Remark~\ref{rem:existence}). For $\{\Gamma_v, \Gamma^c_v \}$, one has
\[
\mathcal{F}_\lambda(\{\Gamma_v, \Gamma^c_v \}) = \lambda \sum_{\partial \Gamma_v}g(e) +(1-\lambda) \sigma(\{\Gamma_v, \Gamma^c_v \})\,,
\]
while for any $2$-partition $\{\tilde \Gamma_k\}$ not containing $\Gamma_v$ one has
\[
\mathcal{F}_\lambda(\{\tilde \Gamma_k\}) = \lambda \sum_{\mathcal{C}(\{\tilde \Gamma_k\}) \cap \partial \Gamma_v }g(e) + \lambda \sum_{\mathcal{C}(\{\tilde \Gamma_k\}) \setminus \partial \Gamma_v} g(e) +(1-\lambda) \sigma(\{\tilde \Gamma_k\})\,.
\]
Our claim corresponds to proving that $\mathcal{F}_\lambda(\{\tilde \Gamma_k\}) - \mathcal{F}_\lambda(\{\Gamma_v, \Gamma^c_v \}) \ge 0$ for some $\lambda$ near to $1$ and for an appropriately small ``perimeter'' of $\Gamma_v$, i.e. $\sum_{\partial \Gamma_v}g(e)$. We have
\begin{align*}
\mathcal{F}_\lambda(\{\tilde \Gamma_k\}) - \mathcal{F}_\lambda(\{\Gamma_v, \Gamma^c_v \}) &= O(1-\lambda) +\lambda \sum_{\mathcal{C}(\{\tilde \Gamma_k\}) \setminus \partial \Gamma_v} g(e)  - \lambda \sum_{\partial \Gamma_v \setminus \mathcal{C}(\{\tilde \Gamma_k\})}g(e)\\
&\ge O(1-\lambda) + \lambda \sum_{\mathcal{C}(\{\tilde \Gamma_k\}) \setminus \partial \Gamma_v} g(e)  -\varepsilon\,.
\end{align*}
As for any $2$-partition not containing $\Gamma_v$ one has that $\mathcal{C}(\{\tilde \Gamma_k\}) \setminus \partial \Gamma_v \neq \emptyset$ the claim follows for $\lambda$ close to $1$ and $\varepsilon<<1$, i.e. the sum of the weights of the edges in $\partial \Gamma_v$.
\end{proof}

\begin{remark}\label{rem:isolated}Theorem~\ref{thm:isolated} is in agreement with the idea that assigning a zero weight to an edge means that the edge is missing. Notice that the hypothesis can be equivalently reformulated by suitably modifying all weights not in $\partial \Gamma_v$ and taking each one of them greater than some $M>>1$.
\end{remark}

The result is somewhat weak as it holds only for $2$-partitions. It would be desirable to extend it to $N$-partitions but it is unclear how to do it, as in the energy expression for a partition containing $\Gamma_v$ would appear the additional term $\lambda \sum_{\mathcal{C}(\{\hat \Gamma_k\}) \setminus \partial \Gamma_v} g(e)$ on which we have no control.

It is as well unclear if fixed any $\lambda>0$, by suitably modifying the perimeter of $\partial \Gamma_v$, one can force the vertex to be isolated, even in the $2$-partition case. Clearly, one would expect that as $\lambda\to 0$, necessarily $\sum_{\partial \Gamma_v} g(e) \to 0$ as well.

\subsection{Refining the number of districts}

Given $j\in \mathbb{N}$, we shall say that a $jN$-partition $\{\hat \Gamma_i\}_{i=1}^{jN}$ is a \emph{$j$-refining} of a $N$-partition $\{\Gamma_k\}_{k=1}^N$ if, up to relabelling, $\{\hat \Gamma_i\}_{i=(k-1)j+1}^{kj}$ is a $j$-partition of $\Gamma_k$. One might wonder if some ``refining'' property holds. For instance, it would be desirable that any (or at least one) minimal $jN$-partition is a refining of a minimal $N$-partition.

This is exactly the situation where you have two different kinds of elections (e.g. Italian and European elections) and one would like to form ``super''-districts for one election just by gluing together some districts of the other election.

Clearly, a necessary condition to allow such a situation is that there exists an optimal $N$-partition such that each of its subgraph has \emph{at least} $j$ vertexes. Even if this necessary condition is satisfied, this refining property can fail, as the following example shows.

\begin{figure}[t]
\begin{center}
\begin {tikzpicture}[auto ,node distance =2.5 cm and 3.5cm ,on grid ,
semithick ,
state/.style ={ circle ,top color =white , bottom color = gray!00 ,
draw,gray , text=black , minimum width =.5 cm}]
\node[state] (1) {$1$};
\node[state] (2) [right =of 1] {$1$};
\node[state] (3) [right =of 2] {$1$};
\node[state] (4) [right =of 3] {$1$};
\node[state] (5) [above =of 1] {$1$};
\node[state] (6) [above =of 2] {$1$};
\node[state] (7) [above =of 3] {$1$};
\node[state] (8) [above =of 4] {$1$};
\path (1) edge [bend right = 0] node[below =0.15 cm] {$4$} (2);
\path (2) edge [bend right = 0] node[below =0.15 cm] {$10$} (3);
\path (3) edge [bend right = 0] node[below =0.15 cm] {$4$} (4);
\path (4) edge [bend right = 0] node[left =0.15 cm] {$1$} (8);
\path (8) edge [bend right = 0] node[below =0.15 cm] {$2$} (7);
\path (7) edge [bend right = 0] node[below =0.15 cm] {$10$} (6);
\path (6) edge [bend right = 0] node[below =0.15 cm] {$2$} (5);
\path (5) edge [bend right = 0] node[left =0.15 cm] {$1$} (1);
\end{tikzpicture}
\caption{A graph where the optimal $4$-partition is not a $2$-refining of the optimal $2$-partition.} \label{fig:refining}
\end{center}
\end{figure}
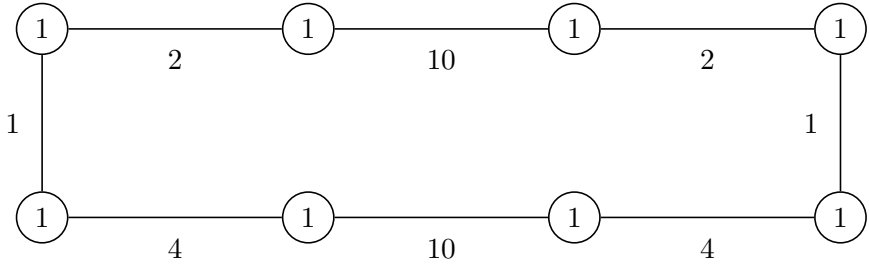

\begin{example}\label{refining} Consider the graph $\Gamma$ of Figure~\ref{fig:refining}. The unique optimal $2$-cut is obtained by removing the two edges with weight $1$. This same partition has zero deviation energy, hence by Theorem~\ref{thm:uniqueness} it is an optimal $2$-partition for every $\lambda$ and the unique one for $\lambda\neq1$.\par

Removing the four edges with weights $2$ and $4$, gives a $4$-partition $D$ with zero deviation energy and cut energy $12\lambda$. Thus, any $4$-partition whose cut set contains an edge of weight $10$ cannot be optimal.\par

Hence, up to symmetries, there is only one \emph{possibly optimal} $4$-partition which is a $2$-refining of the minimal $2$-partition: the one in which are cut one weight $2$ and one weight $4$ edges. This partition $R$ has $\lambda$-energy
\[
\mathcal{F}_\lambda(R)\ =\ 8\lambda+2(1-\lambda)\,,
\]
and it is never minimal. Indeed, either $D$ or $C$, the  partition  whose cut energy is $6\lambda$, is the optimal one. More precisely,
\begin{align*}
\mathcal{F}_\lambda(D)=12\lambda\,, && \mathcal{F}_\lambda(C)=6\lambda+\sqrt6(1-\lambda)\,,
\end{align*}
thus,
\begin{itemize}
\item if $0\leq\lambda<\frac{\sqrt6}{6+\sqrt6}$, $D$ is the optimal $4$-partition;
\item if $\frac{\sqrt6}{6+\sqrt6}<\lambda\leq1$, $C$ is the optimal $4$-partition,
\end{itemize}
with $\bar \lambda = \sqrt{6}/(6+\sqrt{6})$ being the transition value where they are both minimizers.
\end{example}

Even if we assume that a minimal $jN$-partition is a $j$-refining of a minimal $N$-partition, then $\{\hat \Gamma_i\}_{i=(k-1)j+1}^{kj}$ may not be a \emph{minimal}-$j$-partition of $\Gamma_k$, as shown by the following example.

\begin{example} Consider the graph $\Gamma$ in Figure~\ref{fig:strref}, and fix\footnote{\label{l=1/2}The choice $\lambda=\frac12$ has nothing special. For any $\lambda\in(0,1)$ the weights of the graph $\Gamma$ can be suitably chosen to have the same behavior of the shown example: it is sufficient to multiply the weights on vertexes by $\frac1{2\lambda}$ and those on edges by $\frac1{2(1-\lambda)}$ and the energy remains the same.} $\lambda=\frac12$. Let us consider first the $2$-partitions. The optimal $2$-cut ($\lambda=1$), $C_2$, is obtained by removing the two edges of weight $4$. It has energy
\[
\mathcal{F}_\lambda(C_2)\ =\ 8\lambda+6\sqrt{2}(1-\lambda)\,,
\]
and for $\lambda=\frac12$, $\mathcal{F}_{\frac 12}(C_2)\ =\ 4+3\sqrt{2}<\frac{17}{2}$. Hence, all other  $2$-partitions whose cut energy is greater than $17$ cannot be an optimal $2$-partition (w.r.t. the choice $\lambda=\frac 12$). This leaves only two possible partitions with cut energy $16$. Both of them have deviation energy greater than that of $C_2$. Hence, $C_2$ is the optimal $2$-partition of $\Gamma$ w.r.t. the choice $\lambda=\frac12$.

Consider now the $4$-partitions. The optimal $4$-cut ($\lambda=1$) is obtained by removing all edges but one with weight $16$. Among the $2$ possible choices, let $C_4$ be that with smaller deviation energy, i.e. the case where the $14$ and $4$ weight vertexes are grouped together. This $4$-partition has energy
\[
\mathcal{F}_\lambda(C_4)\ =\ 36\lambda+2\sqrt{17}(1-\lambda)\,,
\]
and for $\lambda=\frac12$, $\mathcal{F}_{\frac 12}(C_4)\ =\ 18+\sqrt{17}$. There is only another $4$-partition with smaller deviation energy, $D_4$, in which the $10$ and $4$ weight  vertexes form a district. Its energy is
\[
\mathcal{F}_\lambda(D_4)\ =\ 40\lambda+6(1-\lambda)\,,
\]
and for $\lambda=\frac12$ it is $23$, strictly more than that of $C_4$. Hence, $C_4$ is the optimal $4$-partition for $\lambda=\frac12$.

\begin{figure}[t]
\begin{center}
\begin {tikzpicture}[auto ,node distance =2.5 cm and 3.5cm ,on grid ,
semithick ,
state/.style ={ circle ,top color =white , bottom color = gray!00 ,
draw,gray , text=black , minimum width = .8 cm}]
\node[state] (1) {$20$};
\node[state] (2) [right =of 1] {$14$};
\node[state] (3) [right =of 2] {$4$};
\node[state] (4) [above =of 1] {$20$};
\node[state] (5) [above =of 2] {$10$};
\path (1) edge [bend right = 0] node[left =0.15 cm] {$16$} (4);
\path (4) edge [bend right = 0] node[below =0.15 cm] {$4$} (5);
\path (5) edge [bend right = 0] node[above =0.15 cm] {$12$} (3);
\path (3) edge [bend right = 0] node[below =0.15 cm] {$16$} (2);
\path (2) edge [bend right = 0] node[below =0.15 cm] {$4$} (1);
\end{tikzpicture}
\caption{A graph where the optimal $4$-partition is a $2$-refining of the optimal $2$-partition, but does not induce the optimal $2$-partitions on its components.} \label{fig:strref}
\end{center}
\end{figure}
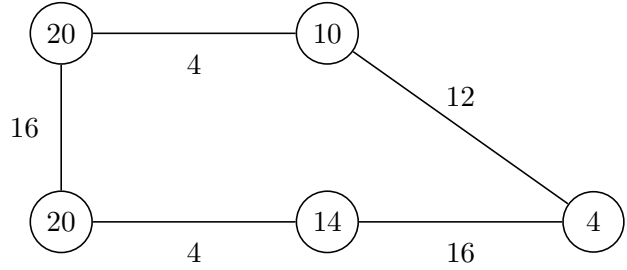
Observe that $C_4$ is a $2$-refining of $C_2$, yet it does not define minimal $2$-partitions for each of the subgraphs of $C_2$. Consider indeed the district with $3$ vertexes of $C_2$. There are two ways to partition it in $2$ subgraphs, one induced by $C_4$, denoted by $\hat C_4$, and one induced by $D_4$, denoted by $\hat D_4$. Their energies are
\begin{align*}
\mathcal{F}_\lambda(\hat C_4)\ =\ 12\lambda+4\sqrt{2}(1-\lambda)\,, && \mathcal{F}_\lambda(\hat D_4)\ =\ 16\lambda\,.
\end{align*}
Relatively to $\lambda=\frac12$, one has
\[
\mathcal{F}_{\frac 12}(\hat C_4)\ =\ 6+2\sqrt{2}\ >\ 8\ =\ \mathcal{F}_{\frac 12}(\hat D_4)\,.
\]
Thus, even though $C_4$ is a $2$-refining of $C_2$ it does not induce minimal-$2$-partitions of the subgraphs of $C_2$.
\end{example}

\section{Other norms for the deviation term}\label{sec:p-norms}

So far we have considered as penalization term the standard deviation from the mean (i.e. the so called $2$-nd central moment), but one could use any $p$-th central moment, defining accordingly the $p$-deviation energy. In principle, fixed $N\in \mathbb{N}$ for any $N$-partition one can think of the $p$-deviation energy as the $p$-norm of the vector ${\bf x}\in\mathbb{R}^N$, where ${\bf x}_k=M(\Gamma_k)-\frac{M(\Gamma)}{N}$, and in general any choice of norm of such a vector can be considered; further, it would be enough to define the penalization term as any function of the vector ${\bf x}\in\mathbb{R}^N$ increasing along half-lines from the origin and symmetric for coordinates swaps and symmetric w.r.t. the origin.\par

It is immediate to see that for any fixed $p$ or any more general deviation energy, Theorem~\ref{thm:uniqueness} remains valid, with the very same proof, since it depends only on linearity of the energy functional on $\lambda$. The same holds true for Theorems~\ref{2vertexes} and~\ref{thm:isolated}.\par

In the following we briefly discuss some stability w.r.t. the choice of the $p$-norm and how to modify the examples to make sure the counterexamples shown above still work for any choice of $p$. We shall denote the $p$-th central moment of $\{\Gamma_k\}$ by $\sigma_p(\{\Gamma_k\})$, which for the sake of completeness we recall to be
\begin{equation*}\label{eq:Vp-energy}
\sigma_p(\{\Gamma_k\})\ =\
\sqrt[p]{\sum_{k=1}^N\left|M(\Gamma_k)-\frac{M(\Gamma)}{N}\right|^p}\,,
\end{equation*}
for $ p\in[1,\infty)$ and
\begin{equation*}\label{eq:Vinfty-energy}
\sigma_\infty(\{\Gamma_k\})=\max_{k}\left|M(\Gamma_k)-\frac{M(\Gamma)}{N}\right|\,,
\end{equation*}
for $p=\infty$. Accordingly, we shall denote by $\mathcal{F}_{\lambda,p}$ the more general functional depending both on $\lambda\in[0,1]$ and on $p\in[1,+\infty]$ given by
\begin{equation*}\label{eq:Lp-energy}
\mathcal{F}_{\lambda,p}(\{\Gamma_k\}) =\lambda P(\{\Gamma_k\}) + (1-\lambda)\sigma_p(\{\Gamma_k\})\,.
\end{equation*}

\subsection{Stability of first transition value w.r.t. \texorpdfstring{$p$}{p}-norms} Notice that for any given $p$, the choice of $\lambda=0$ yields as minimizers the configurations whose $p$-th central moment is smallest. Among these in virtue of the generalization of Theorem~\ref{thm:uniqueness}, we can select one that is a minimizer in the closed interval $\lambda\in[0,\lambda^F(p,\Gamma)]$, with $\lambda^F(p,\Gamma)\in(0,1]$, while it is not for any $\lambda>\lambda^F(p,\Gamma)$. We shall call $\lambda^F(p,\Gamma)$ the \textit{first transition value}. In general, one would expect that by taking greater values of $p$, one would force the penalization term to be more dominant, and this would correspond to having a greater transition value.\par

It is easy to see that the exact opposite happens, at least for $2$-partitions. Indeed, for $N=2$ a partition is uniquely determined by one of its subgraphs and the two subgraphs are such that their deviation from the mean, i.e. $|M(\Gamma_i)- M(\Gamma)/2|$, is the same. Hence, the vector ${\bf x}\in\mathbb{R}^2$ lies in the $1$-dimensional subspace generated by $(1,-1)\in\mathbb{R}^2$ thus all norms are actually the same up to a multiple on this subspace. More precisely,
\begin{align}\label{eq:2-part-p}
\sigma_p(\{\Gamma_k\})\, =\, \sqrt[p]{2} \cdot \sigma_{\infty}(\{\Gamma_k\})\,, && \sigma_q(\{\Gamma_k\})\, =\, \frac{\sqrt[q]{2}}{\sqrt[p]{2}} \cdot \sigma_p(\{\Gamma_k\})\,.
\end{align}
From the above equalities, it immediately follows that $\sigma_p$ is decreasing with respect to $p$ and the first transition value monotonically decreases as $p$ grows, meaning that the deviation optimal $2$-partition is stable for a shorter interval the greater $p$ we choose.

This does not necessarily happen when $N\ge 3$, where we only have the inequalities
\[
\sigma_q(\{\Gamma_k\}) \le \sigma_p(\{\Gamma_k\})\ \leq\ \sqrt[p]{N} \cdot \sigma_{\infty}(\{\Gamma_k\})\,,
\]
for $1\le p<q\le \infty$, from which one can not derive a monotonic behaviour of $\lambda^F(p, \Gamma)$.\par

\subsection{Choosing \texorpdfstring{$\lambda$}{l} and \texorpdfstring{$p$}{p}}
In the previous subsection we introduced the first transition value; in a completely analogous way one can define the \emph{last transition value}, $\lambda^L(p, \Gamma) \in [0,1)$ in such a way that there is a minimal cut configuration which is a minimizer in the closed interval $[\lambda^L(p, \Gamma), 1]$, while it is not for any $\lambda<\lambda^L(p, \Gamma)$. Trivially, unless the minimal cut is as well a minimal deviation, one has the large inequality $\lambda^F(p, \Gamma)\le \lambda^L(p, \Gamma)$. On the one hand, one always wants to consider a parameter $\lambda \le \lambda^L(p, \Gamma)$ in order to enforce some control on the deviation from the mean, otherwise one could end up with completely disproportioned (w.r.t. the mass) subgraphs. It would be desirable to choose $p$ in such a way that the strict inequality $\lambda^F(p, \Gamma)< \lambda^L(p, \Gamma)$ holds. In such a way: there would be a first interval $[0, \lambda^F]$ where relaxing the constraint of having as equal as possible masses does not produce any change in the minimum; there would be another interval $(\lambda^F, \lambda^L)$ where the minima are neither minimal cuts nor minimal $p$-deviations from the mean, i.e. there is some competition between the two energy terms.

\subsection{Generalization of the examples} In view of~\ref{eq:2-part-p} valid for $2$-partitions, Examples~\ref{+-vertex} and~\ref{+-edge} and Remark~\ref{split} in the previous sections remain unchanged, up to modifying the weights of the vertexes by a factor of $\sqrt[p]{2}/\sqrt{2}$ for $p\in[1,+\infty)$ or of $1/\sqrt{2}$ for $p=+\infty$ (see Figures~\ref{fig:+-vertex} and~\ref{fig:+-edge}).\par

In order to generalize Example~\ref{refining} some more effort is needed. We modify the graph of Figure~\ref{fig:refining} by multiplying  some of the weights on the edges by a factor $\alpha=\alpha(p)\geq1$ to be determined later on. The relevant computations are contained in the following example.

\begin{figure}[t]
\begin{center}
\begin {tikzpicture}[auto ,node distance =2.5 cm and 3.5cm ,on grid ,
semithick ,
state/.style ={ circle ,top color =white , bottom color = gray!00 ,
draw,gray , text=black , minimum width =.5 cm}]
\node[state] (1) {$1$};
\node[state] (2) [right =of 1] {$1$};
\node[state] (3) [right =of 2] {$1$};
\node[state] (4) [right =of 3] {$1$};
\node[state] (5) [above =of 1] {$1$};
\node[state] (6) [above =of 2] {$1$};
\node[state] (7) [above =of 3] {$1$};
\node[state] (8) [above =of 4] {$1$};
\path (1) edge [bend right = 0] node[below =0.15 cm] {$4\alpha$} (2);
\path (2) edge [bend right = 0] node[below =0.15 cm] {$10\alpha$} (3);
\path (3) edge [bend right = 0] node[below =0.15 cm] {$4\alpha$} (4);
\path (4) edge [bend right = 0] node[left =0.15 cm] {$1$} (8);
\path (8) edge [bend right = 0] node[below =0.15 cm] {$2$} (7);
\path (7) edge [bend right = 0] node[below =0.15 cm] {$10\alpha$} (6);
\path (6) edge [bend right = 0] node[below =0.15 cm] {$2$} (5);
\path (5) edge [bend right = 0] node[left =0.15 cm] {$1$} (1);
\end{tikzpicture}
\caption{A graph where the optimal $4$-partition is not a $2$-refining of the optimal $2$-partition for suitable choices of $\alpha=\alpha(p)$.} \label{fig:p-refining}
\end{center}
\end{figure}
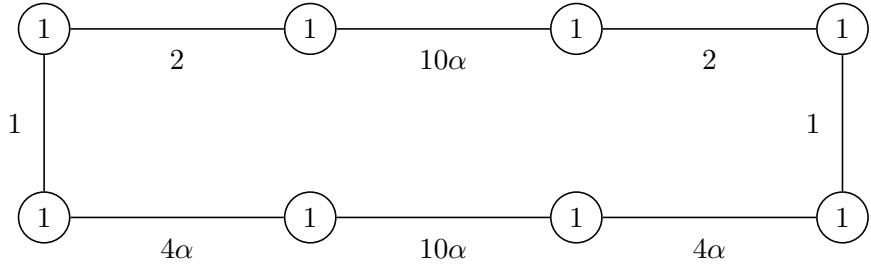

\begin{example}\label{ex:cex_p} Let $p\in[1,+\infty)$ be fixed. Consider the graph $\Gamma$ of Figure~\ref{fig:p-refining}, with $\alpha=\alpha(p)\geq1$ to be determined later on. The unique optimal $2$-cut is obtained by removing the two edges with weight $1$. This same partition has zero deviation energy. Hence by the generalization of Theorem~\ref{thm:uniqueness} to any $p$ discussed at the beginning of Section~\ref{sec:p-norms}, it is an optimal $2$-partition for every $\lambda$ and the unique one for $\lambda\neq1$.\par

Removing the four edges with weights $2$ and $4\alpha$ gives a $4$-partition $D$ with zero deviation energy and cut energy $(4+8\alpha)$. Thus, any $4$-partition whose cut set contains an edge of weight $10\alpha$ cannot be optimal.

Therefore, up to trivial symmetries, there is only one \emph{possibly optimal} $4$-partition which is a $2$-refining of the minimal $2$-partition: the one whose cut set consists of two edges, one of weight $2$ and one of weight $4\alpha$. This partition, which we denote by $I$, has $(\lambda,p)$-energy
\[
\mathcal{F}_{\lambda,p}(I)\ =\ (4+4\alpha)\lambda+\sqrt[p]{4}(1-\lambda)\,.
\]
By choosing $\alpha$ big enough, we can ensure it to be never minimal. Indeed, by calling $C$ the  partition  whose cut energy is $6$, we have
\begin{align*}
\mathcal{F}_{\lambda, p}(D)=(4+8\alpha)\lambda\,, && \mathcal{F}_{\lambda, p}(C)=6\lambda+\sqrt[p]{2+2^p}(1-\lambda)\,.
\end{align*}
Thus,
\begin{align*}
\mathcal{F}_{\lambda,p}(D)<\mathcal{F}_{\lambda,p}(I)\,, && &\text{if $0\leq\lambda< \frac{ \sqrt[p]{4}}{4\alpha + \sqrt[p]{4}}$}\,;\\
 \mathcal{F}_{\lambda,p}(C)<\mathcal{F}_{\lambda,p}(I)\,, && &\text{if $\frac{\sqrt[p]{2+2^p}-\sqrt[p]{4}}{4\alpha-2+\sqrt[p]{2+2^p}-\sqrt[p]{4}}<\lambda\leq1$}\,.
\end{align*}
Therefore by selecting $\alpha$ such that
\[
\alpha >  \frac 12 \cdot \frac{\sqrt[p]{4}}{2\sqrt[p]{4}-\sqrt[p]{2+2^p}}\,,
\]
it is immediate to see that the two intervals above overlap and thus $I$ can never be optimal.
\end{example}

Example~\ref{ex:cex_p} does not work for the choice of the $\infty$-norm, as for $p\to \infty$ the parameter $\alpha$ explodes. Nonetheless, one can construct counterexamples even when choosing such a norm as penalization term, as highlighted in the following example.

\begin{example} Consider the graph $\Gamma$ of Figure \ref{fig:infty-refining} and fix\footnote{As observed in footnote~\ref{l=1/2} on page~\pageref{l=1/2}, it is easy to modify weights and get a counterexample for any $\lambda\in(0,1)$.} $\lambda=1/2$. Being $\Gamma$ a path, symmetric graph it is immediate to see that the $2$-partition $D_2$ with zero deviation is obtained by removing the middle edge, whose weight is $4$; thus, $\mathcal{F}_{\frac12,\infty}(D_2)=\frac 12 \cdot 4$. Since any other $2$-partition has energy at least $\frac52$, $D_2$ is the optimal one.\par

\begin{figure}[t]
\begin{center}
\begin {tikzpicture}[auto ,node distance =2.5 cm and 2cm ,on grid ,
semithick ,
state/.style ={ circle ,top color =white , bottom color = gray!00 ,
draw,gray , text=black , minimum width =.5 cm}]
\node[state] (1) {$2$};
\node[state] (2) [right =of 1] {$2$};
\node[state] (3) [above =of 2] {$4$};
\node[state] (4) [right =of 3] {$4$};
\node[state] (5) [below =of 4] {$2$};
\node[state] (6) [right =of 5] {$2$};
\node[state] (9) [left =of 1] {$2$};
\node[state] (10) [above =of 9] {$2$};
\node[state] (8) [right =of 6] {$2$};
\node[state] (7) [above =of 8] {$2$};
\path (1) edge [bend right = 0] node[below =0.15 cm] {$1$} (2);
\path (2) edge [bend right = 0] node[left =0.15 cm] {$1$} (3);
\path (3) edge [bend right = 0] node[below =0.15 cm] {$4$} (4);
\path (4) edge [bend right = 0] node[right =0.15 cm] {$1$} (5);
\path (5) edge [bend right = 0] node[below =0.15 cm] {$1$} (6);
\path (6) edge [bend right = 0] node[below =0.15 cm] {$1$} (8);
\path (8) edge [bend right = 0] node[left =0.15 cm] {$1$} (7);
\path (9) edge [bend right = 0] node[right =0.15 cm] {$1$} (10);
\path (9) edge [bend right = 0] node[below =0.15 cm] {$1$} (1);
\end{tikzpicture}
\caption{A graph where the optimal $4$-partition is not a $2$-refining of the optimal $2$-partition.} \label{fig:infty-refining}
\end{center}
\end{figure}
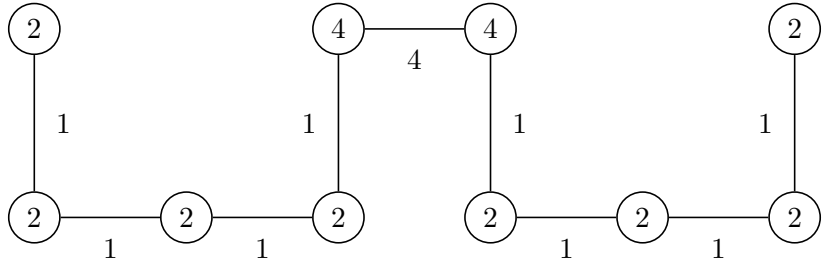

On the one hand, the best $2$-refining of $D_2$ is the $4$-partition $D_4$ with zero deviation, obtained by removing two additional edges of weight $1$ (the third starting from the left and its symmetric); its energy is $\mathcal{F}_{\frac12,\infty}(D_4)=\frac 12 \cdot 6$. On the other hand, consider the (non-symmetric) partition $C_4$ obtained by removing the second, fourth and sixth edge starting from the left; this has deviation energy $2$ (being the masses of the subgraphs $4,~4,~8,~8$) and cut energy $3$, hence $\mathcal{F}_{\frac12,\infty}(C_4)=\frac52$. Thus, the optimal $4$-partition is not induced by the optimal $2$-partition.
\end{example}

\section{Conclusions and future research}\label{sec:conclusions}

The model we propose here is very general and it lays the foundations for future work.  It is truly interesting that as a practical consequence of Theorem~\ref{2vertexes} and of Theorem~\ref{thm:isolated} (specifically with the equivalent hypothesis highlighted in Remark~\ref{rem:isolated}) the best course of action for politicians to achieve some particular partition would be to improve some connections.\par

On the other hand, all ``desired'' properties we checked to fail, do so because the two components of the energies do not interact. This is because the weights $f$ and $g$ can be arbitrarily chosen. It is possible that by forcing some constitutive relation between the edge and the vertex weights these could be ensured.\par

It would be of great interest studying the behaviour of minimizers when the number of vertexes grows. In order to do so, it would be desirable to adopt a $\Gamma$-convergence approach towards a limit continuous model, which is a fundamental tool in the asymptotic behaviour analysis. This approach would provide a bridge between the discrete-continuous models; there is the need of some definition of convergence for sequences of graphs with an increasing number of vertexes. In this direction, a very recent technique has been developed and exploits the so-called \emph{graphons}. These objects are functions on $[0,1]^2$ which somehow represent the adjacency matrices of the graphs (see~\cite{Lov12, LS06}) and their convergence can be studied through a suitable norm, called \emph{cut norm}. Very recently $\Gamma$-convergence of the cut energy has been studied in these terms, see~\cite{BCD19}. We plan to build on these latest results, by expanding the study of the convergence to our more general functional~\ref{eq:L-energy}.

\subsection{Open problems}

There are some open problems we are interested in studying in the future. Among those
\begin{itemize}
\item generalizing the result of Theorem~\ref{2vertexes} trying to force multiple distinct pairs together;
\item generalizing\footnote{Since the first submission of the present paper, this generalization was achieved as part of a forthcoming work jointly with Bertolotti.\label{fn:1}} the result of Theorem~\ref{2vertexes} trying to force $n$ vertexes together;
\item generalizing$^{\ref{fn:1}}$ the result of Theorem~\ref{thm:isolated} to $N$-partitions;
\item trying to either find a modification of the model for which the properties~(i),~(ii) and especially~(v) hold or prove a general impossibility result in the spirit of Arrow~\cite{Arr} and Balinski-Young~\cite{BalYou}.
\item studying the $\Gamma$-convergence of~\ref{eq:L-energy} as $N$ goes to infinity.
\end{itemize}
Clearly there are some necessary hypotheses to be made. Moreover we would expect that forcing multiple vertexes together would require a stringent structure of the subgraph they form, e.g. to be a cycle or wheel graph. It is very possible that these properties do not hold for the fully general model, but we do not have any counterexample to exhibit at the current stage.\par

Finally, we plan to impose some constraints on the choices of $f$ and $g$ by pairing them via some suitable equation. This would imply a more rigid structure to the energy possibly leading to stronger theorems. For instance a possible choice would be to ask
\[
f(v) := \sum_{\partial \Gamma_v} g(e)\,,
\]
being $\Gamma_v$ the subgraph consisting of the lone vertex $v$. This choice would mean, in the politics' application, that each district ensures a flow-in/out of people equal to the number of its citizens. Once such a choice is made or other choices of $g=g(f)$, we are set on doing some numerical simulations reflecting real-world situations occurring in EU states or at the EU level.

\section*{Acknowledgments}
The authors would like to thank the referee for the valuable comments and insights.

\bibliographystyle{plainurl}

\bibliography{Electoral_Districts_bib.bib}

%

\end{document}